\documentclass[12pt]{article}
\usepackage{graphicx} 

\usepackage[french]{minitoc}

\usepackage{float}

\usepackage{amsmath,amsfonts,amssymb,mathrsfs}

\usepackage{color} 

\usepackage{graphics}
\usepackage{graphicx}

\usepackage{geometry} 
\usepackage{enumitem}

\usepackage{pifont}

\newtheorem{theorem}{Theorem}[section]

\newtheorem{proposition}[theorem]{Proposition}
\newtheorem{example}[theorem]{Example}
\newtheorem{remarque}{Remark}
\newtheorem{definition}{Definition}
\newtheorem{proof}[theorem]{Proof}

\newtheorem{notation}{Notation}

\title{New Characterization of  regional controllability and controllability  of deterministic cellular automata via topological and symbolic dynamics notions}

\author{Sara Dridi  \\  University of Setif 1, Algeria \\ email:   sara.dridi@univ-setif.dz}

\begin{document}

\maketitle

\begin{abstract}

This article presents a new characterization of controllability and regional controllability of Deterministic Cellular Automata (CA for short). It focuses on analyzing these problems within the framework of control theory, which have been extensively studied for continuous systems modeled by partial differential equations (PDEs). In the analysis of linear systems, the Kalman rank condition is ubiquitous and has been used to obtain the main results characterizing controllability.
The aim of this paper is to highlight new ways to prove the regional controllability and controllability of CA using concepts from symbolic dynamics instead of using the Kalman condition. Necessary and sufficient conditions are given using the notions of chain transitive, chain mixing, trace approximation, transitive SFT and mixing SFT. Finally, we demonstrate that the presence of visibly blocking words implies that the cellular automaton is not controllable.

\color{black}
\end{abstract}
\textbf{keywords}: Regional controllability, Controllability,  Cellular Automata, Chain transitive, Chain mixing, Trace approximation, Transitive SFT, Mixing SFT, Visibly blocking words. 

\section*{Introduction}
Control theory is an area of applied mathematics that bridges mathematics and technology, dealing with the use of feedback to affect the behaviour of a system in order to achieve a desired goal. The systems under consideration can include differential equations, systems with noise, discrete systems,  or systems with delay, among others. Consequently, control theory integrates various branches of mathematics and is closely tied to engineering challenges. This interdisciplinary study has produced a vast body of literature, with applications spanning fields such as aerospace, biomedical engineering, electrical systems, electronics, economics, and more \cite{intriligator2005applications,palm1975application}.\\
Controllability is a fundamental concept in system analysis, was introduced by R. Kalman in the 1960s and has been significantly developed over the past two centuries. Controllability refers to the ability to design control inputs that guide the system’s state to desired values within a specified time interval $[0, T]$. For linear, finite-dimensional systems, controllability is defined by the well-known Kalman rank condition established by Kalman during this period. The challenge is significantly more complex when dealing with nonlinear and infinite-dimensional systems. Controllability has been investigated for both lumped and distributed parameter systems, which are commonly modeled using partial differential equations with variables and parameters that vary over time and space. These systems include inputs and outputs that enable interaction with their environment.\\
In most control problems studied to date, control inputs are applied across the entire domain of the system. However, for many practical applications, it is more cost-effective and convenient  to achieve specific objectives within only a portion of the domain. This motivated the introduction of new concept that is of  regional controllability  by El Jai and Zerrik \cite{el1995regional} in 1995, focusing primarily on linear distributed parameter systems. This problem has  been further explored in subsequent studies \cite{zerrik2000actuators,ztot2011regional}. \\
 In the context of distributed parameter systems written by partial differential equations (PDEs), the term regional refers to control problems where the desired state is defined and potentially attainable only within a specific part $\omega$ of the overall domain. Regional controllability  is regarded as a specific instance of output controllability. The problem can be  broadly described as determining how to apply suitable controls to steer the system’s state from an initial state to a desired state at a given time $T$ within a subregion $\omega$, which can be located in the interior or on the boundary of the domain. Additionally, the control actions can be implemented within the interior of the domain or along its boundaries.\\
Due to the complexity of distributed parameter systems typically represented by partial differential equations, there is a need for new and more suitable modeling approaches for these systems. In this regard, cellular automata are frequently viewed as a promising alternative to partial differential equations for addressing such complex systems.\\
In this paper, we aim to explore regional controllability and controllability  problems using cellular automata, which are frequently regarded as an effective alternative to partial differential equations.\\
Cellular automata (CA) are among the simplest models for spatially extended systems and can effectively describe complex phenomena. 
 These systems are discrete in terms of space, time, and the states representing physical quantities. Their evolution is driven by simple, local transition rules that may exhibit a complex behaviour. They are used to model various problems in physical, chemical, ecological, and biological phenomena. 
They represent an inherently parallel computational model, consisting of a regular grid of cells, where each cell interacts with its neighbors to determine its next state. In classical synchronous cellular automata, all cells update their values simultaneously in discrete time steps, following deterministic or probabilistic transition rules based on the states of neighboring cells.  They are widely used to study the mathematical properties of discrete systems and to model physical systems. However, controlling systems described by CA presents significant challenges. The techniques used to control discrete systems differ greatly from those used for continuous ones, as discrete systems are typically highly nonlinear, making standard linear approximations inapplicable. In this study, we focus on Boolean deterministic one dimensional CA. \\
A Boolean one dimensional cellular automaton consists of cells that  can be organized in a linear chain. The state of each cell is drawn from a discrete set of values represented by 1 or 0, with $A=\{0,1\}$. The state of each cell changes over time based on the states of its neighboring cells and a local transition function, which defines the connections between them. This function can be deterministic or probabilistic, synchronous or asynchronous, and either linear or nonlinear.\\
Several regional control problems have been explored using cellular automata (CA).  In  \cite{dridi2019markov,dridi2020boundary}, a Markov Chain-based approach was employed to demonstrate the regional controllability of 1D and 2D CA, as an alternative to the traditional Kalman criterion. Another approach, based on graph theory, was also explored, leading to the development of necessary and sufficient conditions for the controllability of 1D and 2D cellular automata \cite{dridi2019graph}. The problems of regional controllability and observability have been addressed using the Kalman condition \cite{dridi2022kalman,el2021some}.\\
In this paper, we continue our investigation into controllability and regional controllability problems of one dimensional deterministic cellular automata. We shall develop a selective review on newly achieved properties linking regional controllability, controllability  and notions of symbolic dynamics. \\
Symbolic dynamics is a branch of mathematics that examines the behavior of sequences formed from a finite alphabet, focusing on their evolution under the action of a shift operator. A key concept in this field is the subshift, a subset of all possible infinite sequences that adhere to specific rules or constraints. Subshifts provide a combinatorial framework for modeling dynamical systems, allowing their evolution to be understood in terms of symbolic patterns. Among these, subshifts of finite type (SFT) which are defined over a finite set of symbols, comprises all infinite (or bi-infinite) symbol sequences that avoid a specified finite set of forbidden words, governed by the action of the shift map, making them computationally manageable while retaining significant structural depth. This dual simplicity and richness position SFTs as a vital link between symbolic dynamics and broader areas of dynamical systems theory. \\
Observing a one-dimensional cellular automaton through a finite window results in a subshift known as a trace subshift. These trace subshifts play a crucial role in analyzing and gaining deeper insights into the dynamics of cellular automata.  Within this framework, the concepts of trace subshift,  chain transitivity and chain mixing provide powerful tools for investigating the regional controllability and controllability of CA. Chain transitivity characterizes a system where, for any two configurations, a sequence of intermediate states can approximate a path connecting them through the system's evolution, highlighting the robustness and interconnected nature of its state space. Chain mixing, a stronger property, ensures that the system can uniformly approximate transitions between any two configurations over time, signifying a higher degree of dynamical unpredictability and uniformity.\\
Following the next section, which reviews essential definitions about topological dynamics, symbolic dynamics and cellular automata, the paper is structured into three main parts as outlined below.

\begin{itemize}
    \item in section 2:  we shall introduce the problem of regional controllability and controllability of CA. 
    \item in section 3: we present the results,  we show that regional controllable CA for every $n$ is automatically  chain transitive (mixing). We prove that regional controllable CA for every $n$ is equivalent to transitive (mixing) SFT for every $n$. Finally, we show that approximation trace is equivalent to regional controllability. 
    \item in section 4: we show that a CA is regional controllable for every $n$ is controllable. Necessary and sufficient conditions are giving using notions: chain transitive (mixing), transitive (mixing) SFT. We demonstrate that the existence of visbibly blocking words implies that CA is not controllable. 
\end{itemize}
Finally conclusion is given.\\

\section{DEFINITIONS AND BACKGROUND}
\subsection{Topological dynamics }
\begin{definition}\cite{kurka2003topological}
A topological dynamical system consists of a pair $(X,F)$ (or simply $F$ when there is no ambiguity), where $X$ is a nonempty compact metric space and $F: X\longrightarrow X$  is a continuous function.
\end{definition}

\begin{definition}\cite{kurka2003topological}
    The orbit of $F$ is a sequence $(x_{i})_{i\in\mathbb{Z}}$ such that : 
    \begin{eqnarray*}
        F(x_{i})=x_{i+1} \quad \forall i\in \mathbb{Z}
    \end{eqnarray*}
\end{definition}
  Let $(X,F)$ and $(Y,G)$ be two dynamical systems. A continuous map $\rho :X\rightarrow Y$ is called a homomorphism if $\rho F=G\rho$. If $\rho$ is surjective, it is referred to as a factor map, and $(Y,G)$ is considered as a factor of $(X,F)$. A subset $H \subseteq X$ is F-invariant if $F(H ) \subseteq H$. Let $d:X \times X \longrightarrow R_{+}\cup \{0\}$ be the metric considered.
\begin{definition}\cite{kurka2003topological}
\begin{enumerate}
    \item   Let $x,y \in X$,  a (finite or infinite ) sequence  $x=x_{0},x_{1},\dots,x_{n}=y\in X$ is an $\varepsilon$-chain from $x$ to $y$ if there exists $n>0$ such that $d(F(x_{i}),x_{i+1})< \varepsilon$, for all $i\in \{0,1,\dots,n-1\}$.
    \item Infinite $\varepsilon$-chain are called $\varepsilon$-pseudo orbits.
    \end{enumerate}
\end{definition}
\begin{definition}\cite{kurka2003topological}
    The dynamical system $(X,F)$
    \begin{enumerate}
        \item is transitive if for all non-empty opens sets $U,V$ there exists $n>0$ such that $F^{n}(U) \cap V \ne \emptyset$. 
         \item is chain-transitive if for all $x,y\in X$ and $\varepsilon >0$ there exists an $\varepsilon$-chain from $x$ to $y$.
        \item is chain-mixing if for all $x,y \in X$ and $\varepsilon>0$ there exists $N>0$ such that for all $n \ge N$ there exists an $\varepsilon$-chain $x=x_{0},x_{1},\dots,x_{n}=y$  from $x$ to $y$.
        \item is mixing if for all non-empty open sets $U, V$ there exists $N > 0$ such that for all $n \ge  N$ it holds that $F^{n}(U)\cap V \ne \emptyset$.
  \item has pseudo-orbit tracing property
        often also called the shadowing property, if for any $\varepsilon> 0$ there exists
$\delta > 0$ such that any finite $\delta$-chain is $\varepsilon$-shadowed by some point. A point $x \in  X$ $\varepsilon$-shadows a finite sequence $x_{0}, x_{1}, \dots, x_{n}$, if for all $i \le n$, $d(F^{i}(x), x_{i}) < \varepsilon$. A (finite or infinite) sequence $(x_{n})_{n\ge 0}$ is a
$\delta$-chain, if $d(F(x_{n}), x_{n+1}) <\varepsilon$ for all $n$. i.e
\begin{equation*}
\forall \varepsilon>0, \exists \delta >0, \forall x_{0},\dots,x_{n}, (\forall i, d(F(x_{i}),x_{i+1}))<\delta)  \implies \exists x, \forall i, d(F^{i}(x),x_{i})<\varepsilon.
\end{equation*}
 \item is minimal, if any its orbit is dense.
    \end{enumerate}
\end{definition}

\subsection{Symbolic dynamics}
\textbf{Configuration space} Let $A$ be a finite set, and let $A^{\mathbb{Z}}$ represent the configuration space of sequences indexed by $\mathbb{Z}$ with values in $A$. An element $s \in A^{\mathbb{Z}}$ is called a configuration which is an infinite sequence made up of elements from $A$. We denote $s(i) = s_{i}$ for $ i \in \mathbb{Z}$. In other words, it is a function $\mathbb{Z} \rightarrow A$. Let $B \subset \mathbb{Z}$ be a finite set and $x \in A^{B}$. The set $[x]=\{s\in A^{\mathbb{Z}} | s_{B}=x\}$ is called a cylinder. When $A$ is endowed with the discrete topology and $A^{\mathbb{Z}}$ with the product topology, cylinders form a countable clopen (both open and closed) basis for this topology. We treat $A^{\mathbb{Z}}$ as a metric space, with the metric defined as:

$$
d(s,e)=\begin{cases}
			2^{-min(|i||s_{i}\ne e_{i})}, & \text{if} \quad s\ne e\\
            0, & \text{otherwise}
		 \end{cases}
$$
For all $s,e\in A^{\mathbb{Z}}$, it is well-known that this metric induces the same topology as previously defined, and that this space is compact.\\
For any $n \in \mathbb{N}/\{0\}$, let $A^{n}$ represent the set of all finite sequences, or finite words, $x=x_{0}x_{1}\dots x_{n-1}$, where each $x_{i}$ is a letter from the set $A$. The length of $x$, denoted $|x|=n$, is the number of letters in the sequence.\\
\textbf{Shift action}  
The two-sided shift (resp one-sided shift), denoted by $\sigma$ which is a shift map $\sigma:A^{\mathbb{Z}}\rightarrow A^{\mathbb{Z}}$,  is defined by $\sigma(s)_{i}=s_{i+1}$ for every for $s=(s_{l})_{l\in \mathbb{Z}} \in A^{\mathbb{Z}} $ and $i \in \mathbb{Z}$ (or for $s=(s_{l})_{l\in \mathbb{N}} \in A^{\mathbb{N}} $ and $i \in \mathbb{N}$. for the one-sided case). This map is a homeomorphism on $A^{\mathbb{Z}}$. 
 
\begin{definition}
A onesided (resp. twosided) subshift is a closed $\sigma$-invariant (resp. strongly) subset  $\Sigma $ of $A^{\mathbb{N}}$ (resp.  $A^{\mathbb{Z}}$ ).
\end{definition}

Here, we present two classic characterizations, one based on languages and the other on graphs.\\
Let $A^{*}=\bigcup\limits_{n\in\mathbb{N}/\{0\}}^{} A^{n}$, if $L\subset A^{*}$  is a language, the set $\Sigma_{L}=\{s\in A^{\mathbb{Z}} | \forall x \in L, x \not\subset s\}$ of configurations that avoid patterns from $L$ is a subshift. Conversely, for any subshift $\Sigma$, we can associate a language $\mathcal{L}(\Sigma)=\{x\in A^{*} | \exists s \in \Sigma, x \subset s \}$, which consists of the finite patterns that appear in some configuration of $\Sigma$. For any length $k \in \mathbb{N}$, the language of order $k$ for $\Sigma$ is $\mathcal{L}_{k}(\Sigma)=\mathcal{L}(\Sigma) \cap A^{k}$. A forbidden language for a subshift $\Sigma$ is a language $L\subset A^{*}$ such that $\Sigma=\Sigma_{L}$. A subshift is called a subshift of finite type (SFT) if it can be described by a finite forbidden language. It is of order $k\in \mathbb{N}$ (k-SFT) if it admits a forbidden language contained in $A^{k}$.\\
A subshift $\Sigma \subset A^{\mathbb{Z}}$ is transitive if, for any two words $u,v \in \mathcal{L}(\Sigma)$  there exists a word
 $w \in \mathcal{L}(\Sigma)$ such that the concatenation $uwv \in \mathcal{L}(\Sigma)$. It is mixing if, for any words $u,v \in \mathcal{L}(\Sigma)$  there exists a natural number $N \in \mathbb{N}$  such that for all $n \ge N$, there exists a word $w \in \mathcal{L}_{n}(\Sigma)$
with length $n$ that satisfies $uwv \in \mathcal{L}(\Sigma)$ \cite{lind2021introduction}.\\
The concept of strongly connected directed graphs is well established. For each subshift of finite type (SFT), there exists a corresponding directed graph and an associated matrix, as described in \cite{brin2002introduction}. This graph may or may not be strongly connected. Let $G=(V,E)$ denote a directed graph (digraph) with vertex set $V$ and the set of directed edges $E$. Let $\mathcal{M}$ be $m\times m$ its associated  adjacency matrix with entries in $\{0,1\}$. Using this matrix, we construct the directed graph $G=(V,E)$. A directed edge $v \rightarrow w$ is included in $E$ if and only if the entry $\mathcal{M}_{v,w}=1$. A digraph is defined as simple if there is at most one edge from any vertex $v$ to any vertex $w$. It is considered strongly connected if there exists a directed path between every pair of vertices.

\subsection{Cellular automata }
A one-dimensional cellular automaton is a parallel, synchronous computation model represented as $(A, m, d, f)$, made up of cells arranged on a regular lattice indexed by $\mathbb{Z}$. Each cell $c_{i}\in\mathbb{Z}$ holds a state $s_{i} $ from a finite set $A$, and its state evolves based on the states of its neighboring cells $s_{(i-m)},\dots,s_{(i-m+d)}$, following a local rule $f:A^{d}\rightarrow A$. Here, $m\in\mathbb{Z}$ is the anchor, and $d>0$ is the diameter of the automaton. If the anchor $m$ is non-negative, it can be set to $0$, making the automaton one-sided, meaning a cell's state is updated only based on its own state and those of its right-hand neighbors. The global function of the automaton, denoted $F:A^{\mathbb{Z}}\rightarrow A^{\mathbb{Z}}$, is defined such that $F^{i}(s)=F(s_{[i-m,i-m+d]})$ for every $s\in A^{\mathbb{Z}}$ and $i\in \mathbb{Z}$.  
When the neighborhood of the automaton is symmetric, the terms "anchor" and "diameter" are replaced by "radius." A cellular automaton with radius $r\in \mathbb{N}/\{0\}$ has an anchor of $r$ and a diameter of $2r+1$.\\
The shift map $\sigma:A^{\mathbb{Z}}\rightarrow A^{\mathbb{Z}}$, is a specific global function of a cellular automaton.  
According to Hedlund's theorem \cite{hedlund1969endomorphisms}, the global functions of cellular automata are precisely the continuous self-maps of $A^{\mathbb{Z}}$ that commute with the two-sided shift $\sigma$. 
For a cellular automaton with state space (alphabet) $A$ and transition function $F$, the space-time diagrams represent two-way infinite orbits. These orbits can be expressed as:
\begin{equation*}
    sp(F)=\{(s^{i})_{i\in\mathbb{Z}}\in (A^{\mathbb{Z}})^{\mathbb{Z}}|\forall i,F(s^{i})=s^{i+1}\}.
\end{equation*}
We can visualize these orbits as colored square lattices, where each row corresponds to points in the orbit and time progresses downwards. \\
Consider a cellular automaton $(A^{\mathbb{Z}},F)$ with radius $r$. In CA, the trace \cite{guillon2008automates} refers to the set of infinite sequences, or words, representing the sequence of states assumed by a specific cells over time see Figure \ref{trace}. The n-trace of $F$ is the one-sided subshift consisting of columns of width $n$ that appear in the space-time diagrams of $F$. It is defined as follow:

\begin{definition}\cite{pierre2010asymptotic}
    The trace application of some totally disconnected discrete dynamical systems
    $(\Sigma \subset A^{\mathbb{Z}}, F)$ in cells $[i,k[$, where $i, k\in \mathbb{Z}$, $i < k$  and $(k-i)+1=n$ is:
    \begin{eqnarray*}
        \mathbb{T}_{F}^{[i,k[}: \Sigma & \longrightarrow & (A^{[i,k[})^{\mathbb{N}} \\
       s  & \longrightarrow& (F^{t}(s)_{[i,k[})_{t \in\mathbb{N}}
   \end{eqnarray*}

    \begin{figure}[H]
        \centering
        \includegraphics[width=15cm,height=8cm]{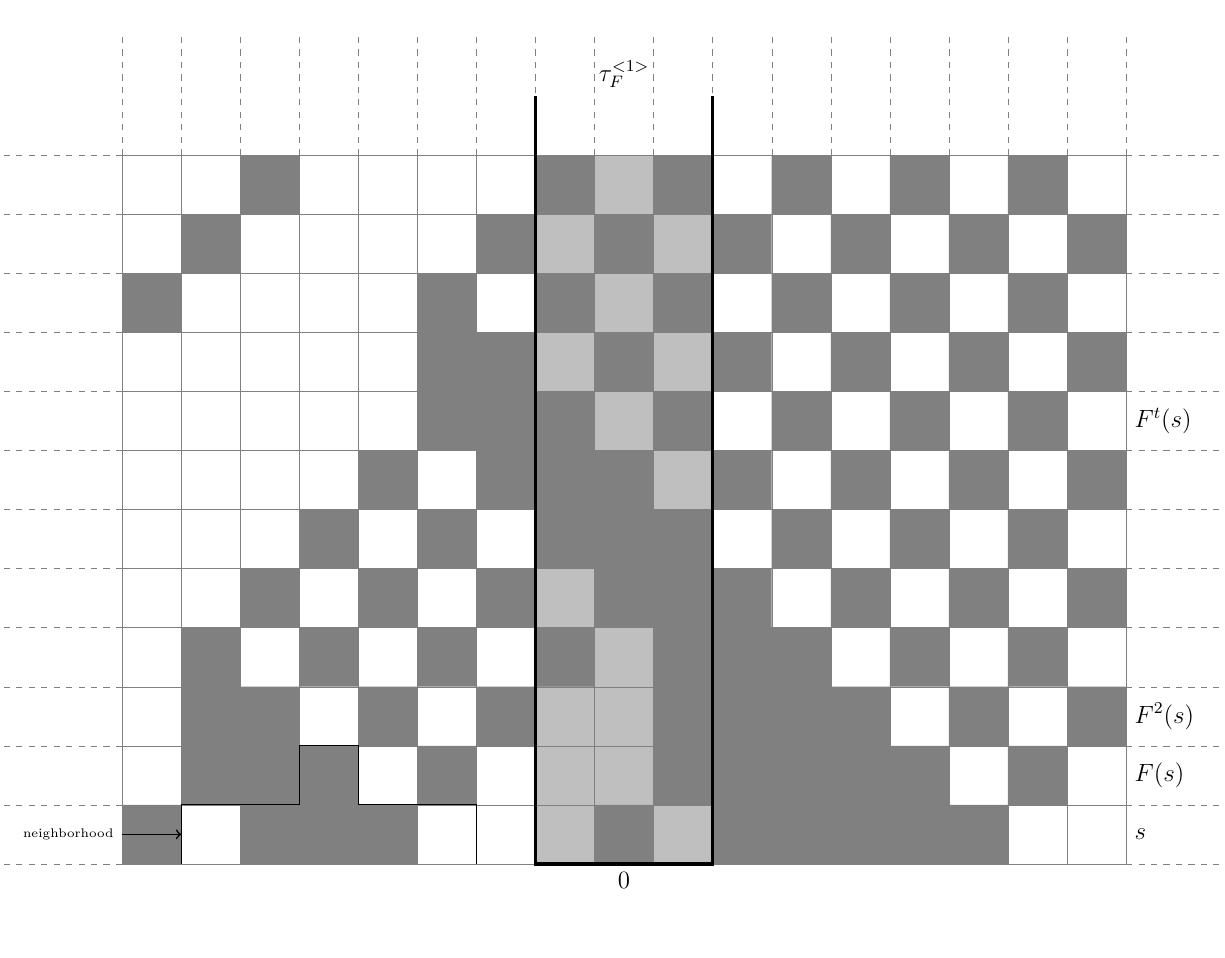}
        \caption{Space-time diagram and trace over segment $[-1,1]$ }
        \label{trace}
    \end{figure}
\end{definition}
The image $\tau_{F}^{[i,k[}=\mathbb{T}_{F}^{[i,k[}(\Sigma)$
is a onesided subshift over alphabet $A^{[i,k[}$.

	\section{Problems statement}
 In this section, we introduce two problems in control theory: controllability and a specific case of it, regional controllability.
 \subsection{Regional Controllability of Cellular automata}
Let us consider a Boolean cellular automaton of radius $r$ defined on a lattice $\Omega$, a region to be controlled $\omega=\{c_{1},\dots,c_{n}\}$ of $n$ cells. With the aim of guiding a dynamical system from an initial state to a selected desirable state in a finite amount of time $T$, we will address the problem of regional controllability via boundary actions on the target region (see Figure \ref{reg-control-CA}).
This issue can be described as follows: \\
Let us consider a target region $\omega = \{c_{1}, \dots, c_{n}\}$ and a one-dimensional CA. The question of regional controllability  is whether it is possible to start from any initial configuration  $s^{0}_{\omega}\in A^{\omega}$ such that $s^{0}_{\omega}= \{s^{0}(1),\dots, s^{0}(n)\}$  and reach any desired configuration $s^{d}_{\omega}\in A^{\omega}$ in the target region at time $T$ by acting on the boundaries  $\{c_{-1}c_{0}, c_{n+1}c_{n+2}\}$ of the controlled region $\omega$ in the time interval $[0, T-1]$ such as: 
\begin{equation*}
    s^{T}(i)=s^{d}(i) \quad \forall i=1,\dots,n 
\end{equation*}

\begin{figure}[H]
\centering
\includegraphics[width=13cm,height=7cm]{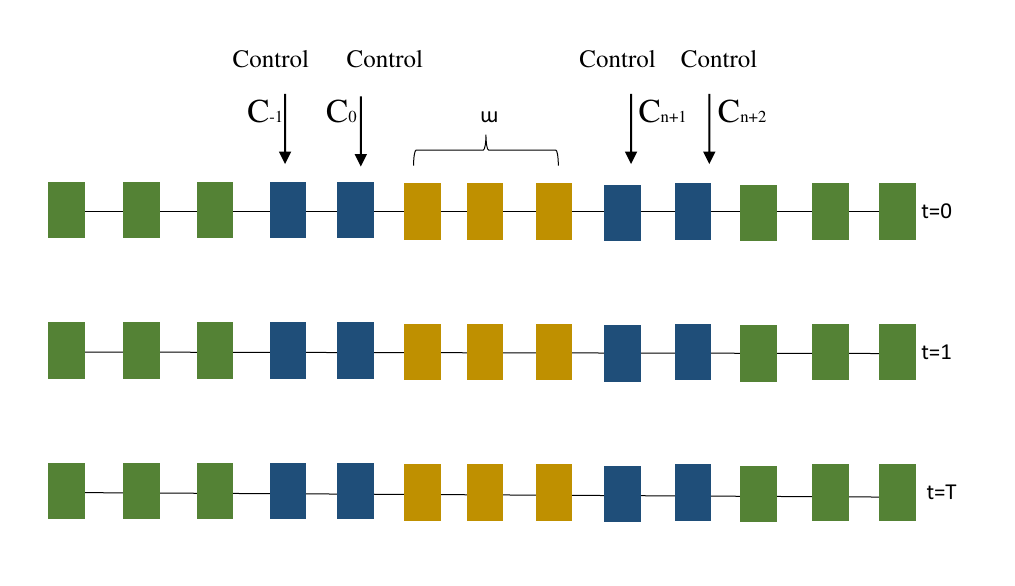}
\caption{ Regional controllability  of one dimensional cellular automaton }
\label{reg-control-CA}
\end{figure}

\begin{notation}
\begin{enumerate}
    \item    $\overline{\omega} = \{c_{-1},c_{0}, c_{1}, \dots, c_{n}, c_{n+1},c_{n+2}\} = \omega \cup  \{c_{-1},c_{0}, c_{n+1}c_{n+2}\}$ where $\{c_{-1},c_{0}, c_{n+1}c_{n+2}\}$ are the boundary cells of $\omega$ where we apply control.
\item $A^{\omega}$ denotes the set words in the controlled region $\omega$. 
\item $A^{\overline{\omega}}$ denotes the set of words in the region $\overline{\omega}$.
\end{enumerate}
\end{notation}

 \subsection{ Controllability of Cellular automata}
We shall deal with the problem of controllability of one dimensional  cellular automata  its aim that is leading a dynamical system from an initial state  to a chosen desired state in a finite time $T$, see Figure \ref{control-CA}.\\

This problem can be defined in this way, let us consider:
\begin{itemize}
    \item a Boolean one-dimensional CA. 
    \item a discrete time horizon $I=[0,T]$.
\end{itemize}

The question is whether it is possible to steer the initial configuration $s^{0}$ in the whole domain (i.e $s^{0}\in A^{\mathbb{Z}}$) to a desired configuration $s^{d}$ such that $s^{d}\in A^{\mathbb{Z}}$ at time $T$ through actions applied on some cells, referred to as $p$ cells (depicted in black). Refer to the following figure for illustration.

\begin{figure}[H]
\centering
\includegraphics[width=13cm,height=7cm]{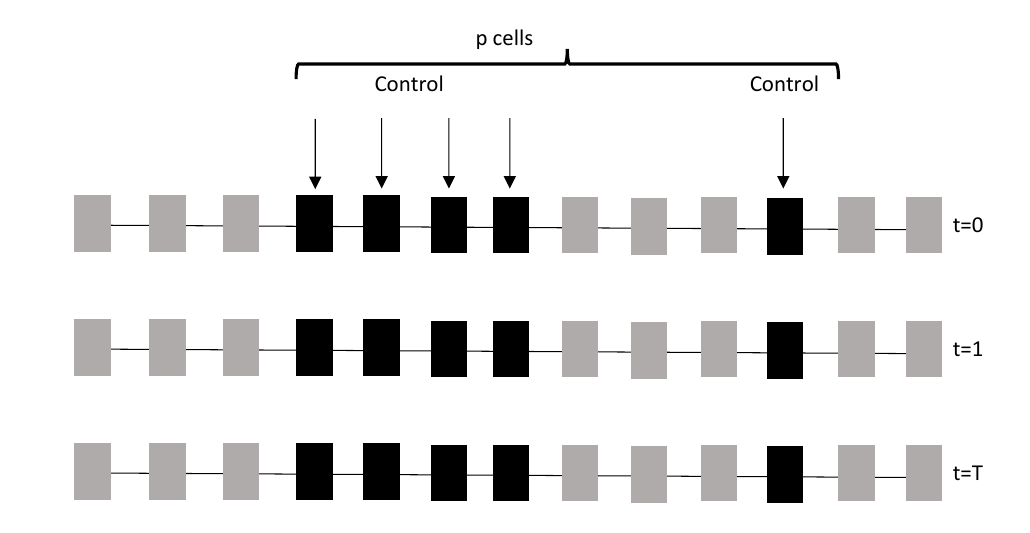}
\caption{ Controllability  of one dimensional cellular automaton }
\label{control-CA}
\end{figure}

\begin{example}
A cellular automaton $(A^{\mathbb{Z}}, F)$ is nilpotent if there exists $q\in A$ and $T\in \mathbb{N}$
 such that for every $s\in A^{Z}$  we have that $F^{T}(s) = \dots qqq \dots$. Then there exists no sequence of controls which allow to start from any configuration and to reach any desired configuration at time $T$ in the target region  and it follows that the CA is not controllable thus not regional controllable.
\end{example}

\begin{definition}
    A point $s \in  A^{\mathbb{Z}}$ of a  dynamical system  $( A^{\mathbb{Z}}, F )$ is equicontinuous, if
    \begin{equation*}
        \forall \varepsilon >0, \exists \delta >0, \forall z \in B_{\delta}(s), \forall n \ge 0, d(F^{n}(z),F^{n}(s))< \varepsilon
        \end{equation*}
  such that $B_{\delta}(s)= \{z \in A^{\mathbb{Z}} | d(s, z) < \delta\}$   When all points of $( A^{\mathbb{Z}}, F )$ are equicontinuity points,  $( A^{\mathbb{Z}},F)$ is said to be equicontinuous.
\end{definition}
\begin{example}The identity rule $204$ (Wolfram rule)\\
$(A^{\mathbb{Z}}, I)$, with $A=\{0,1\}$ where $I(s) = s$, is an equicontinuous CA and it is not controllable thus not regional controllable. 
\end{example}

\section{Characterizing regional controllability for Boolean deterministic CA  using notions of  symbolic dynamics }

\subsection{Transition graph approach and regional controllability problem}\label{sect:transG}

In this part, we recall the transition matrix described in \cite{dridi2019graph}. The evolution of controlled CA for one step can be represented by a directed graph  where the vertices represent the finite  configurations (words) in $A^{n}$ and the arcs  represent the transition from a configuration to another one in one step \emph{i.e.} by applying the global transition function $F$. Consider $(A^{\mathbb{Z}}, F)$ be a cellular automaton with radius $r=2$  where the controlled region $\omega$ is of size $|\omega|=n$ and controls are applied on its two boundary cells  $ \{(c_{-1}c_{0},c_{n+1}c_{n+2})\} $.\\
   We define the transition graph $G_{n}(F)=(V_{n},AR)$ as follows where the vertices $V_{n}$ corresponds to each possible configuration of the region $\omega$ and $AR$ is the set of arcs. Let $v_1$ and $v_2$ be two vertices in $V_{n}$,
   there is an arc from the vertex $v_1$ to the vertex $v_2$ if there exists a boundary control $U_{b}=(x,y)\in
\{(00,00) ; (00,01) \dots  (11,10); (11,11)\}$ 
such that $v_2$ is equal to 
$ F_{|_{A^{\overline\omega}}}(x\cdot
{v_{1}}\cdot y) $ ( $F_{|_{A^{\overline\omega}}}$ is the restriction  of CA to $A^{\overline\omega}$),
where  $v_{1}$ denotes the configuration in the controlled region at time $t$ and $v_{2}$ denotes the configuration in the controlled region at time $t+1$.\\
We denote by $\mathcal{C}$ the transition matrix which is the associate adjacency matrix of the graph $G_{n}$. 
The transition matrix is built as a Boolean matrix of size $ 2^{|\omega|} \times 2^{|\omega|}$. 
There is a $1$ at position $(i,j)$, the $i{\mbox{th}}$ row and $j{\mbox{th}}$ column,  if there is an arc between vertices $i$ and $j$ for all
$i,j$ in $[0:2^{|\omega|}-1]$. Otherwise it stays at $0$.\\
 For each 
vertex $v$, we compute the 16 configurations (represented by 
$w_1, w_2,\dots w_{16}$) obtained by the application of the global transition 
function $F_{|_{A^{\overline\omega}}}$ to the possible configurations 
obtained by the concatenation of the controls 
($(00, 00) ; (00,01) ;\dots ; (11, 11)$) on the extremities of $v$. Then we add  an arc from $v$ to each of the configurations obtained  $w_i$.\\
The following theorem has been derived using the constructed transition matrix $\mathcal{C}$.

\begin{theorem}\cite{dridi2019graph}
    A CA $(A^{\mathbb{Z}},F)$  is regionally controllable for a given rule iff the transition graph
$G_{n}(F)$ associated to the rule has only one Strongly Connected Component (SCC).
\end{theorem}

In the following, we will consider configurations of $(A^{n})^{\mathbb{Z}}$ where $n \ge 0$.
\subsection{Necessary and Sufficient Condition: Chain transitive and Chain mixing}

\begin{notation}
    We introduce the notation $a \rightsquigarrow b$. This means that $b$ is reachable from
$a$ i.e. there is a directed path starting from $a$ to $b$. In other words, there exist vertices
$v_{1}, v_{2}, \dots , v_{i}$ such that $(a, v_{1}), (v_{1}, v_{2}), \dots , (v_{i-1}, v_{i}), (v_{i}, b)$ are arcs in $AR$.
\end{notation}

\begin{definition}
     A finite configuration  $u \in A^{n}$ is reachable from  another finite configuration  $w \in A^{n}$ in one step if there exists a control $(x^{i},y^{i}) \in (A^{r})^{2}$ such that $ F(x^{i}.w.y^{i})=u $. Where $r$ is the radius of the CA. 
\end{definition}

\begin{definition}
    A finite configuration  $u \in A^{n}$ is reachable from  another finite configuration  $w \in A^{n}$ at time $T$ if there exists a sequence of control $(x^{i},y^{i}) \in (A^{r})^{2},\quad \forall i=1,\dots,T-1$ such that: 
    \begin{equation*}
        w \rightsquigarrow w^{1} \rightsquigarrow  \dots \rightsquigarrow u
    \end{equation*}
   
\end{definition}

\begin{example}
   Consider Wolfram's Rule 90, a cellular automaton whose evolution is entirely determined by a rule table. This table maps the next state of a cell based on all possible combinations of its three inputs: $s_{-1}, s_{0}$,  and $s_{+1}$. The next state is calculated as the sum modulo 2 of the states of the cells to the left ($s_{-1}$) and right ($s_{+1}$), expressed as $s_{-1}\oplus s_{+1}$.
   For instance, with $n = 6$ (see Figure \ref{fig-example}), 
   \begin{figure}[H]
       \centering
       \includegraphics[width=10cm,height=2.5cm]{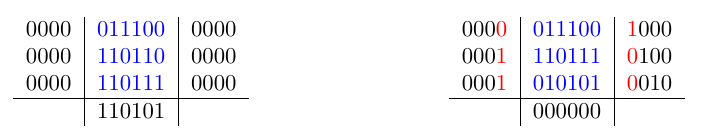}
       \caption{Evolution of Wolfram Rule $90$ Cellular Automaton on the Region $\omega=\{c_{1},\dots,c_{6}\}$ Beginning with the same initial configuration: uncontrolled on the left and controlled on the right }
       \label{fig-example}
   \end{figure}

   Assuming the initial configuration at time $0$ is $\{s^{0}_{1}, s^{0}_{2}, s^{0}_{3}, s^{0}_{4}, s^{0}_{5}, s^{0}
_{6}\} = \{011100\}$ on the region  $\omega= \{c_{1}, \dots , c_{6}\}$ and given a desired null state on $\omega$, a control $U_{b} = ((x^{0},y^{0}), (x^{1},y^{1}), (x^{2},y^{2}))$ where $(x^{0},y^{0})=(0,1), (x^{1},y^{1})=(1,0)$ and $(x^{2},y^{2})=(1,0)$ can be applied on cells $c_{0}, c_{7}$, such that the final CA configuration on $\omega$ obtained at time $T = 3$ from the evolution of rule $90$, is $\{s^{3}_{1}, s^{3}_{2}, s^{3}_{3}, s^{3}_{4}, s^{3}_{5}, s^{3}_{6}\} = \{000000\}$.
\end{example}

\begin{notation}
\begin{enumerate}
    \item  The set of vertices is $V_{n} = A^{n}$.
    \item The set of edges $AR$: For every $v \in  V_{n}$ and $(x^{i},  y^{i}) \in (A^{r})^{2}$  there is a labeled edge from $v$ to $u$ ( $F(x^{i}.v.y^{i}) \xrightarrow{}{} u)$.
    \item The graph $G_{n}(F)$ defines an SFT $S_{n}(F)\subseteq (A^{n})^{\mathbb{N}}$, where the nodes are the finite configurations (words). The points of this SFT are the pseudo orbits of $F$.
    \item  $(v_{1},v_{2})$ is forbidden in this SFT, if there is no edge from   $v_{1}$ to $v_{2}$.
\end{enumerate}
\end{notation}

\begin{definition}
    The CA $(A^{\mathbb{Z}},F)$ is regional controllable on a region $\omega$ if  for every $ s^{0}_{\omega},s^{d}_{\omega} \in A^{\omega}$ \quad
    $
    \exists s^{0}_{\omega},\dots,s^{T}_{\omega}=s^{d}_{\omega}$
  
    such that: 
    \begin{equation*}
s^{0}_{\omega} \rightsquigarrow  s^{d}_{\omega} \quad \text{at time T}
    \end{equation*}
Where  $s^{d}_{\omega}$ is the desired configuration.

\end{definition}

In the following, we use the definition of chain transitive and chain mixing in order to prove regional controllability. 

\begin{proposition}
     A CA $(A^{\mathbb{Z}},F)$ is chain-transitive (mixing)   if and only if CA is regional controllable for every $n$. 
\end{proposition}

\begin{proof}
    \begin{eqnarray*}
&   &  \text{ A CA} \quad  (A^{\mathbb{Z}},F) \quad  \text{is chain-transitive (mixing)  }       \\
&  \iff&  \forall \varepsilon >0 \quad \forall s,s^{d}\in A^{\mathbb{Z}}  \quad  \exists s=s^{0}, \dots, s^{T}=s^{d} \quad d(F(s^{i}),s^{i+1})< \varepsilon \\
  & \iff & \forall k \in \mathbb{N}  \quad  \forall s,s^{d}\in A^{\mathbb{Z}} \quad  \exists s=s^{0},\dots, s^{T}=s^{d}, \quad F(s^{i})_{[-k,k]}= (s^{i+1})_{[-k,k]}\\
  & \iff & \forall k \in \mathbb{N},\quad  G_{k}(F) \quad  \text{is strongly connected } \\
  &\iff & \forall k \in \mathbb{N},\quad  \forall s^{0}_{\omega},s^{d}_{\omega}\in A^{\omega}: \exists s^{0}_{\omega},\dots,s^{T}_{\omega}=s^{d}_{\omega},  \quad s^{0}_{\omega}  \rightsquigarrow  s^{d}_{\omega} \quad \text{at time T}   \\
  & \iff & \forall k \in \mathbb{N}, \quad \text{A CA is regional controllable. }
    \end{eqnarray*}
\end{proof}

\begin{remarque}
    If $(A^{\mathbb{Z}},F)$  is chain-mixing, then $G_{k}(F)$ is not only strongly connected but also possesses the property that there exists an $N \in\mathbb{N}$ such that for every $v,v'\in V_{k}$  and for every $n$ with $n\ge N$, a path of length $n$ exists from $v$ to $v'$.
\end{remarque}

\subsection{Necessary and Sufficient Condition: Transitive SFT and Mixing SFT}
\begin{definition}
    A subshift of finite type (SFT) is called transitive if its associated graph  is strongly connected: there is a sequence of edges from any one vertex to any other vertex. It is precisely transitive subshifts of finite type which correspond to dynamical systems with orbits that are dense. 
\end{definition}

\begin{proposition}
     A CA $(A^{\mathbb{Z}}, F)$ is is regional controllable for every $n$   if and only if $S_{n}(F)$ is transitive for every $n$. 
\end{proposition}

\begin{proof}
    \begin{eqnarray*}
    &  &  \text{ A CA} \quad  (A^{\mathbb{Z}}, F) \quad  \text{is regional controllable for every $n$.   }       \\
      &\iff & \forall k \in \mathbb{N},\quad  \forall s^{0}_{\omega},s^{d}_{\omega}\in A^{\omega}: \exists s^{0}_{\omega},\dots,s^{T}_{\omega}=s^{d}_{\omega},  \quad s^{0}_{\omega}  \rightsquigarrow  s^{d}_{\omega} \quad \text{at time T}   \\
        & \iff & \forall k \in \mathbb{N},\quad  G_{k}(F) \quad  \text{is strongly connected } \\
  & \iff & \forall k \in \mathbb{N},\quad S_{k}(F) \quad \text{is transitive.}
    \end{eqnarray*}
\end{proof}

\begin{definition}
    A subshift of finite type (SFT) is mixing if for any two finite blocks $B$ and $B'$ of symbols (from the alphabet $A$), there exists an integer $N$ such that for all $n\ge N$, there is a sequence in the SFT where block $B$ is followed by block $B'$ after exactly $n$ steps.
\end{definition}

\begin{proposition}
     A CA $(A^{\mathbb{Z}}, F)$ is  regional controllable for every $n$   if and only if $S_{n}(F)$ is mixing for every $n$. 
\end{proposition}

\begin{proof}
    \begin{eqnarray*}
    &  &  \text{ A CA}\quad  (A^{\mathbb{Z}}, F) \quad  \text{is regional controllable.   }       \\
      &\iff & \forall k \in \mathbb{N},\quad  \forall s^{0}_{\omega},s^{d}_{\omega}\in A^{\omega}: \exists s^{0}_{\omega},\dots,s^{T}_{\omega}=s^{d}_{\omega},  \quad s^{0}_{\omega} \rightsquigarrow  s^{d}_{\omega} \quad \text{at time T}   \\
        & \iff & \forall k \in \mathbb{N},\quad  G_{k}(F) \quad  \text{is strongly connected } \\
         &\iff & \forall k \in \mathbb{N},\quad  \forall v_{0},v_{d}\in V_{n} \quad \exists N \in \mathbb{N} \quad  \forall n\ge N  \quad \exists v_{0},\dots,v_{n}=v_{d},  \quad \text{a path of length n}   \\
  & \iff & \forall k \in \mathbb{N},\quad S_{k}(F) \quad \text{is mixing.}
    \end{eqnarray*}
\end{proof}

Let $\mathcal{M}$ be the adjacency matrix with entries $0$ or $1$ related to the SFT. We call the matrix primitive if there exists $M >0$ such that $\mathcal{M}^{M}>0$. 
\begin{proposition}\label{proposmixing}\cite{brin2002introduction}
    An SFT induced by a matrix $\mathcal{M}$ with non zero rows and columns is mixing if and only if $\mathcal{M}$ is primitive. 
\end{proposition}

Based on the result of Proposition \ref{proposmixing}, it can be concluded that:

\begin{proposition}
     A CA $(A^{\mathbb{Z}}, F)$ is  regional controllable   if and only if $\mathcal{M}$ is primitive.  
\end{proposition}

\begin{remarque}
  The time required to achieve regional controllability is $M$, which satisfies the condition that $\mathcal{M}^{M}>0$.
\end{remarque}

\subsection{Necessary and Sufficient Condition: Trace and Approximation Trace}

\begin{definition}\cite{guillon2008automates}
    
The $k-$approximation of a subshift $\Sigma \subset A^{\mathbb{Z}}$ is the largest subshift whose language
of order k is $\mathcal{L}_{k}(\Sigma)$:

\begin{equation*}
    \mathcal{A}_{k}(\Sigma)=\{s\in A^{\mathbb{Z}}| \forall i \in \mathbb{Z}, s_{[i,i+k[}\in \mathcal{L}_{k}(\Sigma)\}
\end{equation*}
Let us note that a configuration belongs to a k-SFT if and only if all of its factors of length $k$ are in the corresponding language. The k-approximation of  $\Sigma$ is therefore the smallest k-SFT containing $\Sigma$.
\end{definition}

\begin{proposition}
    A CA $(A^{\mathbb{Z}}, F)$ is regional controllable if and only if $\mathcal{A}_{2}(\tau _{F}^{n})$ is transitive. 
\end{proposition}

\begin{proof}
\begin{eqnarray*}
 &   & \text{The CA} \quad (A^{\mathbb{Z}}, F) \quad \text{is regional controllable}  \\
   & \iff & \forall  s^{0}_{\omega},s^{d}_{\omega} \in A^{\omega} \quad
 : \exists s=s^{0}_{\omega}, \dots, s^{T}_{\omega}=s^{d}_{\omega} \quad s^{0}_{\omega} \rightsquigarrow  s^{d}_{\omega} \\
   & \iff & \text{The transition graph}\quad 
G_{n}(F)  \quad  \text{associated to the rule has only one SCC}.\\
 & \iff &  \mathcal{A}_{2}(\tau _{F}^{n}) \quad \text{is transitive}. 
\end{eqnarray*}
\end{proof}
\color{black}

\begin{proposition}
     The trace is transitive  if and only if every finite  configuration in $\omega$ is reachable starting from any other finite configuration without applying control.
     \end{proposition}

\begin{proof}
    \begin{eqnarray*}
&   &  \text{ The trace } (\tau_{F}^{n}) \quad  \text{is transitive  }       \\
&  \iff & G_{n}(F) \quad  \text{is strongly connected } \\
& \iff &   \forall s_{\omega},s^{d}_{\omega}\in A^{\omega} \quad  \exists s_{\omega}=s^{0}_{\omega},\dots, s^{T}_{\omega}=s^{d}_{\omega}, F(s^{i})_{[-n/2,n/2]}=(s^{i+1})_{[-n/2,n/2]} \\
&\iff &  \text{every configuration is reachable from any other configuration without applying control.} \\
    \end{eqnarray*}
\end{proof}

Let us now define the controllability problem of CA. 

\section{Characterizing Controllability for Boolean deterministic CA  using notions of  symbolic dynamics } 

Consider the following assumptions \cite{el2003notes}: 
\begin{itemize}
    \item $\Omega$ is a cellular space, 
    \item $I=\{0,1,\dots,T\}$ is a discrete time horizon, 
    
    \item $\Omega_{c_{p}} $ is a subdomain of $\Omega$ formed by  $p$ cells. \\
 
    Let us consider the control space given as follows: 
    
    \begin{equation*}
        \mathcal{U}=C(I \times \Omega_{c_{p}},\mathbb{R})
     \end{equation*}
    It consists of all the bounded controls defined as:

    \begin{eqnarray*}
    U : I \times \Omega_{c_{p}} &\rightarrow &\mathbb{R} \\
           (t,c_{i}) & \rightarrow& U_{t}(c_{i})
    \end{eqnarray*}
    
    Thus, the control operator $G$ is given by: 
    
    \begin{eqnarray*}
    G : \mathcal{U} &\rightarrow & A^{\Omega_{c_{p}}}\\
           U & \rightarrow& GU
    \end{eqnarray*}
    This operator gives the way the control is applied to the CA on the $p$ cells in $A^{\Omega_{c_{p}}}$.

    The maps $G$ associates  to every control variable $U$ a local control function $g: \mathbb{R} \rightarrow  A$ such that: 
    
    	\begin{align}
			 GU_{t}(c)=		\begin{cases}
			 	g(U_{t}(c)) \quad \text{for} \quad c \in \Omega_{c_{p}} \\
			 	0 \quad  \text{elsewhere} 
			 		\end{cases}                
	               \end{align}
	               
    Introducing the characteristic function 
    
    \begin{align}
		\chi_{ \Omega_{c_{p}}}	=		\begin{cases}
			 	1 \quad for \quad c \in \Omega_{c_{p}} \\
			 	0 \quad \text{elsewhere} 
			 		\end{cases}                
	               \end{align}
	               
   one can write, 
    \begin{equation*}
        GU_{t}=g(U_{t}(.)) {\chi_{\Omega_{c_{p}}}}
    \end{equation*}
   
   and so the obtained CA can be considered as a controlled system.

    \begin{definition}\cite{el2003notes}
      A controlled CA can be locally defined by the triple 
      \begin{equation}
          \mathcal{CA}=(f,g,\mathcal{N})
      \end{equation}
      
      where $f$ and $g$ are respectively the local transition and control functions  and $\mathcal{N}$ is the neighbourhood. \\
      
      The global evolution of a controlled CA denoted by   $\mathcal{CA}$ is described in the linear case by the following state equation: 
      
      	\begin{align}
			 		\begin{cases}
			 	s^{t+1}=F(s^{t} +GU_{t})  \\
			 		s^{0} \in A^{\mathbb{Z}} 	\label{equ1CAcontrolled}
			 		\end{cases}                
	               \end{align}
      
      where $F$ is the global dynamics of the autonomous CA and $G$ is the global control function. 
       
    \end{definition}
  
\end{itemize}

\begin{definition}
     An infinite configuration  $s^{t+1} \in A^{\mathbb{Z}}$ is reachable from  another infinite configuration  $s^{t} \in A^{\mathbb{Z}}$ ( and we note $ s^{t} \rightsquigarrow s^{t+1}$) in one step if:  
     \begin{equation*}
        s^{t+1}=F(s^{t} +GU_{t})  
     \end{equation*}
  
\end{definition}

\begin{definition}
    An infinite configuration  $s^{d} \in A^{\mathbb{Z}}$ is reachable from  another infinite configuration  $s \in A^{Z}$ at time $T$ if there exists a sequence of control $U_{i} \quad \forall i=0,\dots,T$ such that: 
    \begin{equation*}
        s=s^{0} \rightsquigarrow s^{1} \rightsquigarrow  \dots \rightsquigarrow s^{T}=s^{d}
    \end{equation*}
   
\end{definition}

\begin{definition}
 A CA is said to be  controllable if every infinite configuration $s'$ is reachable starting from any other infinite configuration $s$ at time $T$.
\end{definition}
\begin{remarque}
      A CA is controllable if and only if  CA is regional controllable for every $n$. 
\end{remarque}
\begin{proof}
     \begin{eqnarray*}
&   &  \text{  A CA is controllable}      \\
&  \iff & \forall  s,s^{d} \in A^{\mathbb{Z}} \quad: \exists s=s^{0}, \dots, s^{T}=s^{d} \quad s  \rightsquigarrow  s^{d} \quad \text{at time } T \\
& \iff & \forall k\in \mathbb{N},  \quad  \forall s,s^{d}\in A^{\mathbb{Z}} \quad  \exists s=s^{0},\dots, s^{T}=s^{d}, \quad F(s^{i})_{[-k,k]}=(s^{i+1})_{[-k,k]} \\
&\iff & \forall k \in \mathbb{N}, \quad G_{k}(F) \quad  \text{is strongly connected } \\
  &\iff & \forall k \in \mathbb{N},\quad  \forall s^{0},s^{d}\in A^{\mathbb{Z}}_{|\omega}: \exists s^{0},\dots,s^{T}=s^{d},  \quad s^{0}  \rightsquigarrow  s^{d} \quad \text{at time T}   \\
&\iff &  \forall k\in \mathbb{N}, \quad \text{A CA is regional controllable.} \\
    \end{eqnarray*}

\end{proof}
It follows that: 
\begin{proposition}
    A CA is controllable if and only if CA is chain transitive.
\end{proposition}

\begin{proposition}
     A CA $(A^{\mathbb{Z}}, F)$ is controllable if and only if CA is chain-mixing. 
\end{proposition}

\begin{proposition}
     A CA $(A^{\mathbb{Z}}, F)$ is controllable   if and only if $S_{n}(F)$ is transitive for every $n$. 
\end{proposition}

\begin{proposition}
     A CA $(A^{\mathbb{Z}}, F)$ is controllable   if and only if $S_{n}(F)$ is mixing for every $n$. 
\end{proposition}

\color{black}

A minimal system is always transitive, and any transitive system is, in turn, chain-transitive \cite{kurka2003topological}. It follows that:

\begin{proposition}
 A CA $(A^{\mathbb{Z}},F)$ is minimal, then $(A^{\mathbb{Z}},F)$ is controllable. 
\end{proposition}

\begin{proposition}
 A CA $(A^{\mathbb{Z}},F)$ is transitive, then $(A^{\mathbb{Z}},F)$ is controllable. 
\end{proposition}

It is a known result that if a cellular automaton possesses the shadowing property and is chain transitive, then it must also be transitive \cite{kurka2003topological}. Thus, it follows that

\begin{proposition}
     If A CA $(A^{\mathbb{Z}}, F)$ has the shadowing property. In particular, if $(A^{\mathbb{Z}}, F)$ is
controllable, then it is transitive.
\end{proposition}

Any chain mixing system that possesses the shadowing property is mixing. It follows that: 

\begin{proposition}
      If A CA $(A^{\mathbb{Z}}, F)$ has the shadowing property. In particular, if $(A^{\mathbb{Z}}, F)$ is
controllable, then it is mixing.
\end{proposition}
To move forward with the next set of results, we require certain concepts related to dynamics.

\begin{definition}\cite{kurka2003topological}
    Let $ p> 0$. A word $x \in A^{*}$ with $|x| \ge  p$ is p-blocking for a CA $(A^\mathbb{Z}, F)$, if there
exists an offset $k \in  [0; |x|-p]$ such that 
\begin{equation*}
    \forall s, s' \in [x]_{0}, \forall t \ge 0  \quad F^{t}(s)_{[k,k+p)}=F^{t}(s')_{[k,k+p)}
\end{equation*}
with $[x]_{0}=\{s \in A^{\mathbb{Z}}: s_{[0,m)}=x\}$

\begin{figure}[H]
    \centering
    \includegraphics[width=13cm,height=3cm]{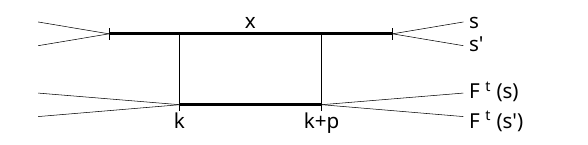}
    \caption{A blocking word }
    \label{blocking-word}
\end{figure}

\end{definition}

\begin{definition}\cite{salo2013color}
    A set of words $W \subset \mathcal{B}_{\ell}(X)$ ($\mathcal{B}_{\ell}(X)$ a set of words of length $\ell $ appearing in configurations of $X$) is visibly blocking for $F$ if
    \begin{itemize}
        \item for all $s \in  X$, if $s_{[0,\ell-1]} \in  W$, iff $F(s)_{[0,\ell-1] } \in  W$, and
        \item for all $s, s' \in  X$ such that $s_{[0,\ell-1]} \in  W$ and $s_{i} = s'_{i}$ for all $i \ge 0$ $(i \le  \ell -1)$,
we have $F^{t}(s)_{i} = F^{t}(s')_{i}$  for all $t \in  \mathbb{N}$ and $i \ge  \ell $ ($i < 0$, respectively).
    \end{itemize}
\end{definition}

\begin{definition}
    \begin{itemize}
          \item A configuration $s$ is called temporally periodic (or just periodic) for $F$ if $F^{p'}(s)=s$ for some period $p' >0$, and it is called (temporally) eventually periodic for $F$ if $F^{m+p'}(s)=F^{m}(s)$ for some pre-period $m \ge 0$ and period $p' >0$. 
        \item A CA is called periodic (eventually periodic) if all configurations are periodic (eventually periodic, respectively).
          \end{itemize}
\end{definition}

\begin{proposition}\label{kurka}\cite{kurka1997languages}
    Let $(A^{\mathbb{Z}}, F)$ be a CA with radius $r > 0$. The following conditions are equivalent.
    \begin{enumerate}
        \item $(A^{\mathbb{Z}}, F)$ is equicontinuous.
        \item There exists $k' > 0$ such that any $x\in A^{2k'+1}$ is $r$-blocking.
        \item There exists a preperiod $m \ge 0$ and a period $p' > 0$, such that $F^{m+p'} = F^{m}$.
    \end{enumerate}

\end{proposition}

The proof of the following proposition builds upon proposition \ref{kurka}.

\begin{proposition}
    Existence of visibly blocking words implies that CA $(A^{\mathbb{Z}},F)$ is not controllable.
    \end{proposition}

\begin{proof}
    Let $W$ be a set of visibly blocking words of length $\ell$ of a CA. By definition $W \ne \emptyset$, so there is a word of length $\ell$ in $W$. We will prove this claim by considering two cases.\\ 
    Case 1:  $W$ contains all words of length $\ell$ then all sufficiently long words are blocking ( words of length $\ell+2r$ are $\ell$-blocking) and so $F$ is eventually periodic. Eventually periodic CA are equicontinuous, therefore are clearly not chain transitive thus not controllable. \\
    Case 2: Assume that there exists a word $v$ of length $\ell$ such that $v\notin W$. Let $[u]$ denote the cylinder set determined by the word $u$, and let $[v]$ denote the cylinder set determined by $v$, both at the same position. Since $v\notin W$, it follows that there is no sequence of allowed transitions that can form a chain from the cylinder $[u]$ to the cylinder $[v]$. Thus $F$ is not chain transitive, clearly not controllable, as required.
  
\end{proof}

\section*{Conclusion}
The problem of controllability is considered as one one of the most prominent problems related to control theory. This work launched  the problems of controllability and 
regional controllability of one dimensional determinitic celular automata using topological and  symbolic dynamics notions.  We have proved the controllability and  regional controllability of CA through an original approach. Firstly,  We have proved that CA that are chain transitive (mixing) are regional controllable for every $n$ and that transitive (mixing) SFT for every $n$  are equivalent to CA that are  regional controllable for every $n$. We have also shown the equivalence between the approximation trace and regional controllability of CA. Finally, we finish by showing that CA that is regional controllable for every $n$ is equivalent to CA that are controllable. Some necessary and sufficient conditions are obtained using the notions: chain transitive ( mixing), transitive (mixing) SFT. Finally, we demonstrate that the existence of visibly blocking words implies that CA is not controllable.

\section*{Acknowledgements}

Special thanks are extended to Pierre Guillon and Ville Salo for their thoughtful comments and valuable feedback, which significantly contributed to the refinement of this work.

\bibliographystyle{acm}

\end{document}